\documentclass[a4paper,12pt]{amsart}
\usepackage[dvipdfmx]{graphicx}
\usepackage{latexsym,amssymb,amsmath,mathrsfs}
\usepackage{amsthm}
\usepackage{mathrsfs}
\usepackage{color}
\usepackage[english]{babel}
\usepackage{comment}
\usepackage[abbrev]{amsrefs}
\usepackage[top=30truemm,bottom=30truemm,left=20truemm,right=20truemm]{geometry}
\newtheorem{theorem}{Theorem}[section]

\newtheorem{proposition}[theorem]{Proposition}
\newtheorem{corollary}[theorem]{Corollary}
\theoremstyle{remark}
\newtheorem{remark}[theorem]{Remark}
\DeclareMathOperator{\Ric}{Ric}

\newcommand{\e}{\mathrm{e}}
\newcommand{\Hess}{\mathrm{Hess}\,}
\renewcommand{\d}{\mathrm{d}}

\title[Steklov eigenvalue estimates for affine connections]{A Steklov eigenvalue estimate for affine connections and its application to substatic triples}
\author{Yasuaki Fujitani}
\keywords{Steklov eigenvalue, weighted Ricci curvature, substatic condition}

\address{Department of Mathematics, Osaka University, Toyonaka, Osaka, 560-0043, Japan \linebreak (Present address: Graduate School of Mathematical Sciences, The University of Tokyo, Komaba, Tokyo, 153-8914, Japan)} 
\email{yasuakifujitani@g.ecc.u-tokyo.ac.jp}
\begin{document}
\maketitle
\begin{abstract}
Choi-Wang obtained a lower bound of the first eigenvalue of the Laplacian on closed minimal hypersurfaces.
On minimal hypersurfaces with boundary,
Fraser-Li established an inequality giving a lower bound of the first Steklov eigenvalue as a counterpart of the Choi-Wang type inequality.
These inequalities were shown under lower bounds of the Ricci curvature.
In this paper,
under non-negative Ricci curvature associated with an affine connection introduced by Wylie-Yeroshkin,
we give a generalization of Fraser-Li type inequality.
Our results hold not only for weighted manifolds under non-negative $1$-weighted Ricci curvature but also for substatic triples.
\end{abstract}
\section{Introduction}
For an $n$-dimensional Riemannian manifold $(M,g)$ and $f \in C^{\infty}(M)$,
we consider the \textit{weighted measure} $\mu := \e^{-f}v_g$,
where $v_g$ is the Riemannian volume measure.
The triple $(M,g,f)$ is called a \textit{weighted Riemannian manifold}.
For $N \in (-\infty,1]\cup [n,\infty]$,
we define the \textit{$N$-weighted Ricci curvature} by
\begin{align*}
    \Ric_f^N := \Ric + \Hess f - \frac{\d f\otimes \d f}{N - n},
\end{align*}
where we only consider a constant function $f$ if $N = n$,
and the last term vanishes if $N = \infty$.
We have 
\begin{align*}
    \Ric_f^\infty \leq \Ric_f^1.
\end{align*}
Hence,
we see that the condition $\Ric_f^1 \geq K g$ is weaker than the condition $\Ric_f^\infty \geq K g$.
In the weighted case with $N = \infty$,
Wei-Wylie \cite{WW} obtained a Bishop-Gromov type volume comparison theorem and Fang-Li-Zhang \cite{FLZ} obtained a Cheeger-Gromoll type splitting theorem for the case $N =\infty$ (see also \cite{L,WW}). 
Later,
in the case $N = 1$,
Wylie \cite{W2} obtained a splitting theorem of Cheeger-Gromoll type and Wylie-Yeroshkin \cite{WY} obtained a volume comparison theorem of Bishop-Gromov type.
Moreover,
for $\varphi := \frac{f}{n-1}$,
they introduced an affine connection:
\begin{align*}
    \nabla^\varphi_X Y := \nabla_X Y - \d \varphi(X)Y - \d \varphi(Y)X.
\end{align*}
We call this \textit{Wylie-Yeroshkin type affine connection}.
Once we have an affine connection,
we may define the Ricci curvature associated with it (see e.g., \eqref{eq:affine-ricci}), 
which we call the \textit{affine Ricci curvature}.
Wylie-Yeroshkin \cite{WY} revealed that $\Ric_f^1$ coincides with the affine Ricci curvature associated with $\nabla^\varphi$.
Later,
Li-Xia \cite{LX} gave a further generalization of $\nabla^\varphi$ and introduced an affine connection:
\begin{align}\label{eq:def-affine}
    D^{\alpha,\beta}_X Y := \nabla_XY - \alpha \d f(X)Y - \alpha \d f(Y)X + \beta g(X,Y)\nabla f
\end{align}
for $\alpha ,\beta \in \mathbb{R}$.
We see that $D^{\frac{1}{n-1},0}$ coincides with the Wylie-Yeroshkin type affine connection.
This enlightened a relationship between the $1$-weighted Ricci curvature and the substatic condition:
\begin{align}\label{eq:sustatic-condition}
    V \Ric - \Hess V + (\Delta V)g \geq 0
\end{align}
with positive $V\in C^{\infty}(M)$.
Indeed,
the affine Ricci curvature $\Ric^{D^{0,1}}$ associated with $D^{0,1}$
satisfies
\begin{align*}
    \Ric^{D^{0,1}} = \Ric - \frac{\Hess V}{V} + \frac{\Delta V}{V}g
\end{align*}
for $V := \e^{f}$.
The right-hand side is called the \textit{static Ricci tensor}.
We see that the non-negativity of the static Ricci tensor implies the substatic condition.
If $((M,g),V)$ satisfies \eqref{eq:sustatic-condition},
it is called a \textit{substatic triple}.
Examples of substatic triples include \textit{deSitter-Schwarzschild manifold} and \textit{Reissner-Nordstr\"{o}m manifold} (see e.g., Brendle \cite{B3}).
Recently,
some comparison geometric properties for substatic triples such as a volume comparison theorem and a splitting theorem were obtained by Borghini-Fogagnolo \cite{BF}.
In addition,
they obtained an isoperimetric inequality by using a Willmore type inequality on non-compact substatic triples.
In \cite[Appendix]{BF},
they pointed out that the non-negativity of the $1$-weighted Ricci curvature is equivalent to the substatic condition after a suitable conformal change.
Indeed,
after the conformal change,
$D^{0,1}$ can also be regarded as a Wylie-Yeroshkin type affine connection (see e.g., Proposition \ref{prop:dual-relation}).
Hence,
the static Ricci tensor can be considered as the $1$-weighted Ricci curvature for some weighted manifold,
which yields the same conclusion as in \cite[Appendix]{BF}.

In this paper,
we investigate lower bounds of the first Steklov type eigenvalue on hypersurfaces with boundary under non-negative affine Ricci curvature associated with Wylie-Yeroshkin type affine connection.
Since this condition is equivalent to $\Ric_f^1\geq 0$ and also to the substatic condition \eqref{eq:sustatic-condition},
our results also hold true for both weighted Riemannian manifolds under $\Ric_f^1\geq 0$ and substatic triples.
As an application,
we also show a compactness theorem for hypersurfaces with boundary in smooth topology.

The Steklov eigenvalue estimate for hypersurfaces with boundary in this paper can be seen as a counterpart of the Choi-Wang type inequality for hypersurfaces without boundary.
Here,
we first introduce Choi-Wang type inequalities.
For Riemannian manifolds under lower bounds of Ricci curvature,
a lower bound of the first eigenvalue of the Laplacian on minimal hypersurfaces was first obtained by Choi-Wang \cite{CW}.
As an application,
Choi-Schoen \cite{CS} showed a compactness theorem for minimal hypersurfaces.
Moreover,
a compactness theorem for self-shrinkers was obtained by Colding-Minicozzi \cite{CM3}.
After that,
Ding-Xin \cite{DX} gave a further generalization of them.

These results on Choi-Wang type inequalities have been generalized to those in weighted settings.
For $(M,g,f)$,
the Laplacian is generalized to the \textit{weighted Laplacian} as follows:
\begin{align*}
    \Delta_f = \Delta - g(\nabla f,\nabla \cdot).
\end{align*}
For an immersed hypersurface $\Sigma$ and a unit normal vector field $\nu$ on $\Sigma$,
the mean curvature is generalized to the \textit{weighted mean curvature}:
\begin{align*}
    H_{f,\Sigma} = H_\Sigma - f_\nu,
\end{align*}
where
$f_\nu := g(\nabla f,\nu)$ and $H_\Sigma$ is the mean curvature on $\Sigma$.
We say that $\Sigma$ is \textit{$f$-minimal} if $H_{f,\Sigma} \equiv0$.
It should be noted that self-shrinkers in Euclidean spaces are $f$-minimal hypersurfaces if we take an appropriate function as $f$.
In the context of Choi-Wang type inequalities in the weighted setting,
the first eigenvalue of the weighted Laplacian on $f$-minimal hypersurfaces has been investigated under lower bounds of $\Ric_f^N$.
In the weighted case with $N = \infty$,
Li-Wei \cite{LW} obtained a Choi-Wang type inequality for compact manifolds (see also Ma-Du \cite{MD}),
and they generalized a Choi-Schoen type compactness theorem.
After that,
also in the case $N = \infty$,
Cheng-Mejia-Zhou \cite{CMZ} showed it for non-compact manifolds.
In the case $N = 0$,
a further generalization was obtained by \cite{FS}.
As far as we know,
any Choi-Wang type inequality for the case $N = 1$ has not yet been obtained.

We now turn to the first Steklov eigenvalue estimate.
There is a growing interest in minimal hypersurfaces with boundary.
Especially,
the research on minimal hypersurfaces without boundary has been generalized to those for minimal hypersurfaces with free boundary.
In particular,
under non-negative Ricci curvature,
Fraser-Li \cite{FL} 
showed a Choi-Schoen type compactness theorem for minimal hypersurfaces with free boundary.
Instead of the Choi-Wang inequality,
they showed a lower bound of the first Steklov eigenvalue,
and used it to show the compactness theorem.
We note that an isoperimetric inequality was also used in \cite{FL}.
These results were also generalized to those in weighted settings under non-negative $\Ric_f^N$.
On an immersed hypersurface $\Sigma$ with boundary,
the $f$-Steklov eigenvalue problem is as follows:
\begin{align}\label{eq:steklov-problem}
    \begin{cases}
        \Delta_{f,\Sigma} \,u = 0 &\mbox{ on } \Sigma,\\
        u_\nu = \lambda u &\mbox{ on }\partial \Sigma,
    \end{cases}
\end{align}
where $\nu$ is the outer unit normal vector field on $\partial \Sigma$.
Barbosa-Wei \cite{BW} generalized inequalities of Fraser-Li type to the weighted case with $N = \infty$.
In particular,
they obtained a lower bound of the first $f$-Steklov eigenvalue.
The aim of this paper is to generalize results in \cite{BW} to those under non-negative Ricci curvature associated with Wylie-Yeroshkin type affine connections.

On weighted manifolds,
while the Choi-Wang inequality has not yet been generalized to the case $N = 1$,
we obtain Fraser-Li type inequalities in the case $N = 1$.
On substatic triples,
another type of the Steklov eigenvalue problem is known in Huang-Ma-Zhu \cite{HMZ},
which is different from \eqref{eq:steklov-problem}.
Our Fraser-Li type inequality also gives a lower bound of the first eigenvalue of the boundary value problem in \cite{HMZ}.

We review organizations.
In section \ref{sec:tool},
we prepare tools for subsequent sections.
In particular,
we present a Reilly formula (Proposition \ref{prop:d-reilly}) and a second variation formula for the area (Proposition \ref{prop:second-variation-area-1-Ricci}).
Also,
we address the relation between the static Ricci tensor and the affine Ricci curvature associated with the Wylie-Yeroshkin type affine connection (Proposition \ref{prop:dual-relation}).
In section \ref{sec:isoperimetric},
we show a Fraser-Li type isoperimetric inequality (Theorem \ref{thm:FL-isoperimetric}),
and show the existence of minimal hypersurfaces with free boundary as a byproduct (Corollary \ref{cor:BW-existence}).
Furthermore,
we explicitly write down a Fraser-Li type isoperimetric inequality for substatic triples (Corollary \ref{cor:substatic-isoperimetric}).
In section \ref{sec:frankel},
we obtain a Frankel type property (Proposition \ref{prop:free-frankel}).
In section \ref{sec:steklov},
as an application of results in previous sections,
we give a lower bound of the first Steklov eigenvalue (Theorem \ref{thm:steklov-1-ricci}) and obtain a compactness theorem (Theorem \ref{thm:compactness}) of Fraser-Li type.
As applications,
we give a lower bound of the first Steklov eigenvalue associated with $D^{0,1}$ (Corollary \ref{cor:steklov-substatic}) and a compactness theorem for minimal surfaces in substatic triples (Corollary \ref{cor:substatic-compactness}).
These results also hold true under $\Ric_f^1\geq 0$,
and can be regarded as a generalization of results in \cite{BW} to the weighted case $N = 1$.

\section{Several formulas}\label{sec:tool}
We show several formulas which are useful in the following sections.
Let $(M,g)$ be an $n$-dimensional compact Riemannian manifold with boundary and $\varphi\in C^\infty(M)$.
We set $D := \nabla^\varphi$ and $\mu := \e^{-(n-1)\varphi}v_g$.
We denote the outer unit normal vector field on $\partial M$ by $\nu(\partial M)$.
For $\nu := \nu(\partial M)$,
we set
\begin{align*}
    \mathrm{II}_{\partial M}(X,Y) = g( \nabla_X \nu,Y),\quad  H_{\partial M} = \mathrm{tr}\mathrm{II}_{\partial M}.
\end{align*}
Li-Xia \cite{LX} defined the \textit{$D$-mean curvature} by 
\begin{align*}
    H_{\partial M}^{D} = H_{\partial M} - (n-1)\varphi_{\nu}.
\end{align*}
Let $\Sigma$ be an immersed hypersurface.
For a unit normal vector field $\nu(\Sigma)$ on $\Sigma$,
we define $\mathrm{II}_\Sigma$,
$H_{\Sigma}$ and $H^D_{\Sigma}$ in the same manner as above.
We say $\Sigma$ is \textit{$D$-minimal} if $H^{D}_{\Sigma}\equiv 0$.
We define the \textit{$D$-Riemannian curvature tensor} and \textit{$D$-Ricci curvature} by 
\begin{align}\label{eq:affine-ricci}
    R^D (X,Y)Z &:= D_X D_Y Z - D_Y D_X Z - D_{[X,Y]}Z,\quad \Ric^D (X,Y) := \sum_{i = 1}^n g(R^D(X,E_i)E_i,Y),
\end{align}
where $\{E_i\}_{i = 1}^n$ is an orthonormal frame of the tangent bundle.
%
For $D$-Ricci curvature,
we have the following Reilly formula:
\begin{proposition}\label{prop:d-reilly}
Let $(M,g)$ be an $n$-dimensional compact Riemannian manifold and $\varphi \in C^{\infty}(M)$.
Also,
let $\Omega$ be a compact set with piecewise smooth boundary $\partial \Omega = \cup_{i = 1}^l \Sigma_i$,
and $S := \cup_{i = 1}^l \partial \Sigma_i$.
For $\phi \in C^0(\Omega) \cap C^{\infty}\left( \Omega \backslash S \right)$,
we assume that there exists a constant $C > 0$ such that 
\begin{align*}
    \| \phi\|_{C^3(\Omega')} \leq C
\end{align*}
for any set $\Omega'$ in the interior of $\Omega\backslash S$. 
Then we have
\begin{align*}
    &\int_{\Omega}  \left\{ \left(\Delta \phi - ng( \nabla \varphi, \nabla \phi )\right)^2 - \left| \mathrm{Hess}\, \phi - g( \nabla \varphi, \nabla\phi ) \right|^2 - \Ric^D\left( \nabla \phi, \nabla \phi \right) \right\}\ \d \mu\\
    &= \sum_{i = 1}^l  \int_{\Sigma_i} \left( H^D_{\Sigma_i}\,\phi_{\nu(\Sigma_i)}^2 + \mathrm{II}_{\Sigma_i} \left( \nabla_{\Sigma_i}\,\psi,  \nabla_{\Sigma_i}\,\psi \right) \right) \ \d \mu_{\Sigma_i}\\
    &\,\,\quad + \sum_{i = 1}^l \int_{\Sigma_i} \left(  \phi_{\nu(\Sigma_i)}  \Delta_{(n-1)\varphi, \Sigma_i} \psi  -  g_{\Sigma_i}\left(\nabla_{\Sigma_i}\,\psi, \nabla_{\Sigma_i}\, \phi_{\nu(\Sigma_i)}   \right) \right) \ \d \mu_{\Sigma_i},
\end{align*}
where we set $\psi := \phi|_{\partial M}$,
$D := \nabla^\varphi$ and $\mu := \e^{-(n-1)\varphi}v_g$.
\end{proposition}
\begin{proof}
By direct calculations,
we have
\begin{align*}
    &\left(\Delta \phi - ng(\nabla \varphi , \nabla \phi)\right)^2 - \left| \mathrm{Hess}\, \phi - g( \nabla \varphi, \nabla\phi ) \right|^2 - \Ric^D\left( \nabla \phi, \nabla \phi \right)\\
    &\,\,\quad = (\Delta_f \phi)^2 - \left| \Hess \phi \right|^2 - \Ric_f^\infty (\nabla \phi, \nabla\phi),
\end{align*}
where $f := (n-1)\varphi$.
Together with the Reilly formula for the case $N = \infty$ on $(M,g,f)$ (see e.g., \cite[Proposition 3.1]{BW}),
we arrive at the desired inequality.
\end{proof}
\begin{remark}
This also follows from the Reilly formula in Li-Xia \cite[Theorem 3.6]{LX}.
We also refer to \cite[Proposition 2.5]{FS}.
\end{remark}
For $\phi \in C^{\infty}(\Sigma)$ and $\nu := \nu(\Sigma)$,
let $\Sigma_t$ be the normal variation associated with $\phi \nu$ such that $\Sigma_0 = \Sigma$.
If we have
\begin{align*}
    \frac{\d^2}{\d t^2}\bigg|_{t = 0} \mu_{\Sigma_t}(\Sigma_t) \geq 0
\end{align*}
for any $\phi\in C^{\infty}(\Sigma)$,
we say that $\Sigma$ is \textit{$D$-stable}.
Otherwise,
$\Sigma$ is called \textit{$D$-unstable}.
Also,
$\Sigma$ is said to be \textit{properly} immersed if $\partial \Sigma$ is contained in $\partial M$,
and $\Sigma$ is a hypersurface with \textit{free boundary} if $\Sigma$ meets $\partial M$ orthogonally along $\partial \Sigma$.
We are now in a position to give the following second variation formula:
\begin{proposition}\label{prop:second-variation-area-1-Ricci}
Let $(M,g)$ be an $n$-dimensional compact Riemannian manifold with boundary and $\varphi\in C^{\infty}(M)$.
Also,
let $\Sigma$ be a compact properly immersed two-sided $D$-minimal hypersurface in $M$ with free boundary.
For $\phi \in C^{\infty}(\Sigma)$ and $\nu := \nu(\Sigma)$,
let $\Sigma_t$ be the normal variation of $\Sigma$ associated with $\phi\nu$ such that $\Sigma_0 = \Sigma$.
Then
\begin{align*}
    \left.\frac{\d^2}{\d t^2}\right|_{t=0} \mu_{\Sigma_t}(\Sigma_t) 
    &= \int_\Sigma \left\{ |\nabla_{\Sigma} \,\phi|^2 - \left( \Ric^D(\nu,\nu) + \left|\mathrm{II}_\Sigma - \varphi_\nu g_\Sigma \right|^2 \right)\phi^2 \right\} \ \d \mu_{\Sigma}\\
    &\,\,\quad - \int_{\partial\Sigma} \mathrm{II}_{\partial M} (\nu,\nu)\phi^2 \ \d \mu_{\partial\Sigma},
\end{align*}
where $D := \nabla^\varphi$ and $\mu := \e^{-(n-1)\varphi}v_g$.
\end{proposition}
\begin{proof}
From Castro-Rosales \cite[Proposition 3.5]{CR},
we have 
\begin{align}\label{eq:second-var-1}
    &\left.\frac{\d^2}{\d t^2}\right|_{t=0} \mu_{\Sigma_t}(\Sigma_t) \\
    &\,\,\quad = \int_\Sigma \left\{ |\nabla_{\Sigma} \,\phi|^2 - \left( \Ric_{(n-1)\varphi}^\infty(\nu,\nu) + |\mathrm{II}_\Sigma|^2 \right)\phi^2 \right\} \ \d \mu_{\Sigma} - \int_{\partial\Sigma} \mathrm{II}_{\partial M} (\nu,\nu)\phi^2 \ \d \mu_{\partial\Sigma}\nonumber.
\end{align}
Since $\Sigma$ is $D$-minimal, 
we have $H_{\Sigma} = (n-1) \varphi_\nu$.
Hence, 
we find
\begin{align*}
    \left| \mathrm{II}_\Sigma - \varphi_\nu g_{\Sigma} \right|^2
    &=  |\mathrm{II}_\Sigma|^2 - 2 \varphi_\nu \, \mathrm{tr}\mathrm{II}_\Sigma + (n-1)\varphi_\nu^2 \\
    &= |\mathrm{II}_\Sigma|^2 -  2\varphi_\nu\, H_{\Sigma}  + (n-1) \varphi_\nu^2 \\
    &= |\mathrm{II}_\Sigma|^2 - (n-1)\varphi_\nu^2 .
\end{align*}
This leads us to 
\begin{align*}
    \Ric_{(n-1)\varphi}^\infty(\nu,\nu)  + |\mathrm{II}_\Sigma|^2 = \Ric_{(n-1)\varphi}^1(\nu,\nu) + \left| \mathrm{II}_\Sigma - \varphi_\nu g_{\Sigma} \right|^2 = \Ric^D + \left| \mathrm{II}_\Sigma - \varphi_\nu g_\Sigma \right|^2.
\end{align*}
Substituting this into \eqref{eq:second-var-1},
we complete the proof.
\end{proof}
As is mentioned in the introduction,
for $f = (n-1)\varphi$,
we have 
\begin{align}\label{eq:relation-1-Ricci}
    \Ric^D = \Ric_f^1,\quad H_\Sigma^D = H_{f,\Sigma}.
\end{align}
The second identity implies that $\Sigma$ is $D$-minimal if and only if it is $f$-minimal.
In a similar way,
we note that there is a further relation between $D$ and the substatic condition as follows:
\begin{proposition}\label{prop:dual-relation}
Let $(M,g)$ be a Riemannian manifold,
$\varphi \in C^{\infty}(M)$ and $\Sigma$ be an immersed hypersurface in $M$.
For $\widetilde{g} := \e^{-2\varphi}g$,
we denote the Levi-Civita connection by $\widetilde{\nabla}$ and we set $D^* := \widetilde{\nabla}^{-\varphi}$.
Then for $V := \e^\varphi$,
we have
\begin{align}\label{eq:relation-substatic}
    \Ric^{D^*}_{\widetilde{g}} = \Ric - \frac{\Hess V}{V} + \frac{\Delta V}{V}g,\quad H^{D^*}_{\widetilde{g}, \Sigma} = V H_{\Sigma}.
\end{align}
where $\Ric^{D^*}_{\widetilde{g}}$ is the $D^*$-Ricci curvature and $H^{D^*}_{\widetilde{g},\Sigma}$ is the $D^*$-mean curvature for $\widetilde{g}$.

In particular,
$\Sigma$ is a $D^*$-minimal hypersurface in $(M,\widetilde{g})$ if and only if $\Sigma$ is a minimal hypersurface in $(M,g)$.
\end{proposition}
\begin{proof}
It is noted by Yeroshkin \cite[Proposition 2.4]{Yero} that 
\begin{align*}
    D^*_X Y = \nabla_X Y + g( X,Y )\nabla \varphi.
\end{align*}
From Li-Xia \cite[Propostion 2.3]{LX}, 
we have the first equality.
It follows from the direct calculation (see e.g., \cite[(2.6)]{BW}) that 
\begin{align*}
    \e^{\varphi}\mathrm{II}_\Sigma(e_i,e_j) = \widetilde{\mathrm{II}}_\Sigma(\tilde{e}_i,\tilde{e}_j) + \widetilde{g}(\widetilde{\nabla}\varphi, \tilde{\nu}) \, \widetilde{g}(\tilde{e}_i,\tilde{e}_j),
\end{align*}
where $\nu := \nu(\Sigma)$ and $\{e_i\}_{i = 1}^{n-1}$ is an orthonormal frame on the tangent bundle of $(\Sigma,g_\Sigma)$,
and we set $\tilde{e}_i := \e^{\varphi}e_i$ and $\tilde{\nu} := \e^\varphi \nu$.
This implies the second identity.
\end{proof}
\begin{remark}
The first identity \eqref{eq:relation-substatic} coincides with the conclusion in Borghini-Fogagnolo \cite[Appendix]{BF}.
\end{remark}
In sections below,
we obtain results under $\Ric^D \geq 0$ with $D := \nabla^\varphi$.
It follows immediately from the relation \eqref{eq:sustatic-condition},
\eqref{eq:relation-1-Ricci} and \eqref{eq:relation-substatic} that our results also hold true even under $\Ric_f^1\geq 0$ or the substatic condition.

\section{Isoperimetric inequality}\label{sec:isoperimetric}
In this section,
we show a Fraser-Li type isoperimetric inequality.
For a Riemannian manifold $(M,g)$ and a geodesic $\gamma : [0,d]\rightarrow M$,
the \textit{index form} is defined by 
\begin{align*}
    I(X,X) := \int_0^d \left( |X'(t)|^2 - g(R(X,\gamma'(t))\gamma'(t),X) \right)\ \d t,
\end{align*}
where $R$ denotes the Riemannian curvature tensor on $(M,g)$.
Wylie \cite{W1} obtained the following formula for the index form (see \cite[Proposition 5.1]{W1}):
\begin{proposition}[\cite{W1}]\label{prop:index}
Let $(M,g)$ be a complete Riemannian manifold,
$\gamma : [0,d]\rightarrow M$ be a geodesic and $\varphi \in C^{\infty}(M)$.
For $D := \nabla^\varphi$ and a vector field $X$ perpendicular to $\gamma'$,
we have 
\begin{align*}
    I (X,X) = \int_0^d \left( |X'(t) - g(\nabla\varphi,\gamma'(t)) X|^2 - g(R^D(X,\gamma'(t))\gamma'(t),X)  \right)\ \d t + \left[g( \nabla\varphi, \gamma'(t))|X(t)|^2\right]_{0}^d.
\end{align*}
\end{proposition}
We have the following isoperimetric inequality of Fraser-Li type:
\begin{theorem}\label{thm:FL-isoperimetric}
Let $(M,g)$ be an $n$-dimensional compact Riemannian manifold with boundary and $\varphi \in C^{\infty}(M)$.
For $D := \nabla^\varphi$,
we assume 
\begin{align*}
    \Ric^D \geq 0, \quad H^D_{ \partial M} > 0.
\end{align*}
Then there is no closed embedded $D$-minimal hypersurface.
Let $\Sigma$ be an immersed $D$-minimal hypersurface in $M$.
If $3\leq n \leq 7$,
then there exists a constant $c > 0$,
depending only on $(M,g)$ and $\varphi$,
such that 
\begin{align*}
    v_{\widetilde{g},\Sigma}(\Sigma) \leq c v_{\widetilde{g},\partial \Sigma}(\partial \Sigma),
\end{align*}
where we set $\widetilde{g} := \e^{-2\varphi} g$.
\end{theorem}
\begin{proof}
As for the first statement,
we give a proof by contradiction.
We assume that there exists a closed embedded $D$-minimal hypersurface $\Sigma$ in $M$.
We have $\Sigma \cap \partial M = \emptyset$ since $H^D_{\partial M} > 0$ and $H^D_{\Sigma} \equiv 0$.
Let $d := d (\Sigma, \partial M)$ and $\gamma : [0,d]\rightarrow M $ be a minimizing geodesic from $\Sigma$ to $\partial M$ parametrized by the arclength. 
By the second variation formula for length together with Proposition \ref{prop:index},
we have
\begin{align*}
    0 &\leq 
    -\int_0^d \e^{2\varphi(\gamma(t))} \Ric^D(\gamma'(t),\gamma'(t))\ \d t - \e^{2\varphi(\gamma(d))}H^D_{\partial M}(\gamma(d)) + \e^{2\varphi(\gamma(0))}H^D_{\Sigma}(\gamma(0)),
\end{align*}
where $H^D_\Sigma$ is the $D$-mean curvature for \textcolor{black}{$\gamma'(0)$} (see also \cite[Proposition 3.6]{FS}).
This leads to a contradiction. 

We turn to the second statement.
By direct calculations (see e.g., \cite[(2.6)]{BW}),
we have $H_{\widetilde{g},\partial M} = \e^\varphi H_{\partial M}^D > 0$.
Also,
the first statement implies that $(M,\widetilde{g})$ contains no closed embedded minimal hypersurface.
Hence,
we may apply White \cite[Theorem 2.1]{White2} to $(M,\widetilde{g})$,
and conclude the proof.
\end{proof}
\begin{remark}
We refer to Fraser-Li \cite[Lemma 2.2]{FL} for the unweighted case $f\equiv 0$ and Barbosa-Wei \cite[Lemma 2.1]{BW} for the weighted case with $N = \infty$.
Since we have the relation \eqref{eq:relation-1-Ricci},
this is the generalization of them to the case $ N = 1$.
\end{remark}
Together with Proposition \ref{prop:dual-relation},
this yields an isoperimetric inequality for substatic triples:
\begin{corollary}\label{cor:substatic-isoperimetric}
Let $((M,g),V)$ be an $n$-dimensional compact substatic triple with boundary.
We assume $H_{\partial M} > 0$.
Then there is no closed embedded minimal hypersurface in $M$.
Let $\Sigma$ be an immersed minimal hypersurface in $M$.
If $3\leq n \leq 7$,
then there exists a constant $c > 0$, 
depending only on $((M,g),V)$,
such that 
\begin{align*}
    v_{g,\Sigma}(\Sigma) \leq c v_{g,\partial \Sigma}(\partial \Sigma).
\end{align*}
\end{corollary}
\begin{proof}
We set $\varphi := \log V$ and $\widetilde{g} := \e^{-2\varphi} g$.
For $\widetilde{g}$,
we denote the Levi-Civita connection by $\widetilde{\nabla}$ and we set $D^* := \widetilde{\nabla}^{-\varphi}$.
By Proposition \ref{prop:dual-relation},
we see 
\begin{align*}
    \Ric^{D^*}_{\widetilde{g}} \geq 0, \quad {H}^{D^*}_{\widetilde{g},\Sigma} = \e^{\varphi} H_\Sigma = 0,\quad {H}^{D^*}_{\widetilde{g},\partial M} = \e^\varphi H_{\partial M} > 0.
\end{align*}
We apply the argument in Theorem \ref{thm:FL-isoperimetric} to $(M,\widetilde{g})$ and $D^*$,
and arrive at the desired assertion.
\end{proof}
Lastly,
we provide an existence property.
Indeed,
as an application of Theorem \ref{thm:FL-isoperimetric},
we have the following result:
\begin{corollary}\label{cor:BW-existence}
Let $(M,g)$ be a three-dimensional compact weighted Riemannian manifold with boundary and $\varphi\in C^\infty(M)$.
For $D := \nabla^\varphi$,
we assume 
\begin{align*}
    \Ric^D \geq 0,\quad H^D_{\partial M} > 0. 
\end{align*}
Then there exists a properly embedded $D$-minimal surface with free boundary.
\end{corollary}
\begin{proof}
It follows from Theorem \ref{thm:FL-isoperimetric} that $M$ does not contain any closed embedded $D$-minimal surface.
For $\widetilde{g}:= \e^{-2\varphi}g$,
this implies that $(M,\widetilde{g})$ does not contain any closed minimal surface.
By applying Li \cite[Theorem 1.1]{LMM},
we see that there exists a properly embedded minimal surface with free boundary in $(M,\widetilde{g})$.
We complete the proof.
\end{proof}
\begin{remark}
Barbosa-Wei \cite[Theorem 1.1]{BW} obtained the weighted case with $N = \infty$.
This generalizes it to the case $N = 1$.
\end{remark}
\section{Frankel Property}\label{sec:frankel}
In this section,
we provide a Frankel property for manifolds with boundary.
First,
we present a topological property:
\begin{proposition}\label{prop:unstable}
Let $(M,g)$ be an $n$-dimensional compact Riemannian manifold with boundary and $\varphi \in C^{\infty}(M)$.
For $D := \nabla^\varphi$ and $k > 0$,
we assume
\begin{align*}
    \Ric^D \geq 0,\quad \mathrm{II}_{\partial M} \geq k\,g_{\partial M},\quad H^D_{\partial M} > 0.
\end{align*}
Let $\Sigma$ be a two-sided properly immersed $D$-minimal hypersurface with free boundary.
Then $\Sigma$ is $D$-unstable.
Furthermore, 
if $M$ is orientable,
then $H_{n-1}(M,\partial M)$ vanishes.
\end{proposition}
\begin{proof}
We set $\mu := \e^{-(n-1)\varphi}v_g$.
For $\phi \in C^\infty(\Sigma)$ and $\nu := \nu(\Sigma)$,
let $\Sigma_t$ be the normal variation associated with $\phi \nu $ with $\Sigma_0 = \Sigma$.
By applying Proposition \ref{prop:second-variation-area-1-Ricci} to $\phi \equiv 1$,
we see 
\begin{align*}
    \frac{\d^2}{\d t^2} \bigg|_{t = 0} \mu_{\Sigma_t}(\Sigma_t) \leq  -k\,\mu_{\partial\Sigma}(\partial \Sigma).
\end{align*}
Then $\Sigma$ is $D$-unstable.

Next,
we assume $H_{n-1}(M,\partial M) \neq 0$ and see that this leads to a contradiction.
By the argument in \cite[Lemma 2.1]{FL},
we take a properly embedded two-sided $D$-stable $D$-minimal hypersurface $\Sigma$ with free boundary.
Here,
we note that it is enough to consider the case $\Sigma$ is smooth.
If $\partial \Sigma \neq \emptyset$,
we have a contradiction with the statement above.
If $\partial \Sigma = \emptyset$,
it contradicts with Theorem \ref{thm:FL-isoperimetric}.
Therefore,
we have $H_{n-1}(M,\partial M) = 0$.
\end{proof}
\begin{remark}
We refer to Fraser-Li \cite[Lemma 2.1]{FL} for the unweighted case $f\equiv 0$ and Barbosa-Wei \cite[Lemma 2.2]{BW} for the weighted case with $N = \infty$.
This is the generalization of them to the case $N = 1$.
\end{remark}
As an application,
we have the following Frankel type property:
\begin{proposition}\label{prop:free-frankel}
Let $(M,g)$ be a compact Riemannian manifold with boundary and $\varphi \in C^{\infty}(M)$.
For $D := \nabla^\varphi$ and $k > 0$,
we assume 
\begin{align*}
    \Ric^D \geq 0, \quad \mathrm{II}_{\partial M} \geq k\,g_{\partial M}, \quad H_{\partial M}^D \geq 0.
\end{align*}
Let $\Sigma_1$ and $\Sigma_2$ be properly \textcolor{black}{embedded} orientable $D$-minimal hypersurfaces with free boundary in $M$.
Then $\Sigma_1$ and $\Sigma_2$ must intersect.
\end{proposition}
\begin{proof}
We give a proof by contradiction.
We assume that $\Sigma_1$ and $\Sigma_2$ do not intersect.
From Proposition \ref{prop:unstable},
we see $H_{n-1}(M,\partial M) = 0$,
where $n$ is the dimension of $M$.
Then there exists a compact connected domain $\Omega$ such that $\partial \Omega = \Sigma_1 \cup \Sigma_2 \cup \Gamma$ with a set $\Gamma$ contained in $M$.
Let $u$ be the solution of the following boundary value problem:
\begin{align*}
    \begin{cases}
        \Delta u - n g( \nabla \varphi, \nabla u )= 0 &\mbox{ on }\Omega,\\
        u = 0 & \mbox{ on }\Sigma_1,\\
        u = 1 & \mbox{ on }\Sigma_2,\\
        u_\nu = 0& \mbox{ on }\Gamma,
    \end{cases}
\end{align*}
where $\nu := \nu(\Gamma)$.
By Proposition \ref{prop:d-reilly},
we have 
\begin{align*}
    0 \geq \int_{\Omega} \Ric^D \left( \nabla u, \nabla u \right)\ \d \mu + \int_{\Gamma} \mathrm{II}_{\Gamma}(\nabla_{\Gamma}\, z, \nabla_{\Gamma}\, z) \ \d \mu_{\Gamma},
\end{align*}
where $z := u|_{\Gamma}$ and $\mu := \e^{-(n-1)\varphi} v_g$.
This implies that $u$ is constant,
which contradicts with $u = 0$ on $\Sigma_1$ and $u = 1$ on $\Sigma_2$.
This concludes the proof.
\end{proof}
\begin{remark}
We refer to Fraser-Li \cite[Lemma 2.4]{FL} for the unweighted case $f\equiv 0$ and Barbosa-Wei \cite[Lemma 2.1]{BW} for the weighted case with $N = \infty$.
This is the generalization of them to the case $N = 1$.
\end{remark}
This implies the following property:
\begin{corollary}\label{cor:dividedness}
Let $(M,g)$ be a compact Riemannian manifold with boundary and $\varphi \in C^{\infty}(M)$.
For $D := \nabla^\varphi$ and $k > 0$,
we assume 
\begin{align*}
    \Ric^D \geq 0, \quad \mathrm{II}_{\partial M} \geq k\,g_{\partial M}, \quad H^D_{\partial M} > 0.
\end{align*}
Let $\Sigma$ be a properly embedded orientable $D$-minimal hypersurface in $M$ with free boundary.
Then $\Sigma$ divides $M$ into two components.
\end{corollary}
\begin{proof}
We set $U := M \backslash \Sigma$ and $U^{*} := U \cup \partial U$.
We assume $U^{*}$ is connected.
Wylie \cite[Corollary 4.6]{W2} implies that $\partial U$ is connected.
On the other hand,
since $\Sigma$ is orientable and connected by Proposition \ref{prop:free-frankel},
we see that $\partial U$ has two components.
This leads us to a contradiction.
Hence, 
we arrive at the desired assertion.
\end{proof}
\begin{remark}
We refer to Fraser-Li \cite[Corollary 2.10]{FL} for the unweighted case $f \equiv 0$ and Barbosa-Wei \cite[Corollary 2.5]{BW} for the weighted case with $N = \infty$.
This is the generalization of them to the case $N = 1$.
\end{remark}
Another application of Proposition \ref{prop:unstable} is the following topological property:
\begin{corollary}\label{cor:diffeo-ball}
Let $(M,g)$ be a three-dimensional compact weighted Riemannian manifold with boundary and $\varphi\in C^\infty(M)$.
For $D := \nabla^\varphi$ and $k > 0$,
we assume 
\begin{align*}
    \Ric^D \geq 0, \quad \mathrm{II}_{\partial M} \geq  k \,g_{\partial M}, \quad H^D_{\partial M} > 0.
\end{align*}
Then $M$ is diffeomorphic to a three-dimensional ball.
\end{corollary}
\begin{proof}
By the argument in \cite[Theorem 1.2]{BW},
it is enough to consider the case $M$ is orientable.
By Theorem \ref{thm:FL-isoperimetric},
there is no closed embedded $D$-minimal hypersurface.
For $\widetilde{g} := \e^{-2\varphi}g$,
this implies that $(M,\widetilde{g})$ contains no closed embedded minimal hypersurface.
It follows from the argument in Meeks-Simon-Yau \cite[Theorem 5]{MSY} that $M$ is a handlebody.
Since Proposition \ref{prop:unstable} implies $H_2(M,\partial M) = 0$,
$M$ has no handle.
Therefore,
$M$ is diffeomorphic to a three-dimensional ball.
\end{proof}
\begin{remark}
We refer to Fraser-Li \cite[Theorem 2.11]{FL} for the unweighted case $f\equiv 0$ and Barbosa-Wei \cite[Theorem 1.2]{BW} for the weighted case $N = \infty$.
This is the generalization of them to the case $N = 1$.
\end{remark}
\section{Steklov eigenvalue estimate and its application}\label{sec:steklov}
Contrary to the case of the Choi-Wang inequality,
our Fraser-Li type inequality is obtained in the case $N = 1$ as we observe below.
Indeed,
we show the following eigenvalue estimate:
\begin{theorem}\label{thm:steklov-1-ricci}
Let $(M,g)$ be an $n$-dimensional compact orientable Riemannian manifold with boundary and $\varphi \in C^{\infty}(M)$.
For $D := \nabla^\varphi$ and $k > 0$,
we assume 
\begin{align*}
    \Ric^D \geq 0, \quad \mathrm{II}_{\partial M} \geq k\,g_{\partial M}, \quad H^D_{\partial M} > 0.
\end{align*}
Let $\Sigma$ be a properly embedded $D$-minimal hypersurface in $M$ with free boundary and $\lambda_{1,\Sigma}^{\mathrm{Ste}}$ be the first $(n-1)\varphi$-Steklov eigenvalue in \eqref{eq:steklov-problem}.
If $\Sigma$ is orientable or $\pi_1(M)$ is finite,
we have
\begin{align*}
    \lambda_{1,\Sigma}^{\mathrm{Ste}} \geq \frac{k}{2}.
\end{align*}
\end{theorem}
\begin{proof}
We first consider the case $\Sigma$ is orientable.
By Corollary \ref{cor:dividedness},
we see that $\Sigma$ divides $M$ into two components.
We choose one component, 
and denote it by $\Omega$.
We have $\partial \Omega = \Sigma \cup \Gamma$ for a set $\Gamma$ in $\partial M$.
Here,
we see $\partial \Sigma = \partial\Gamma$.
Let $z$ be an eigenfunction of $\lambda_{1,\Sigma}^{\mathrm{Ste}}$,
and $\phi$ be the solution of the following boundary value problem:
\begin{align*}
    \begin{cases}
        \Delta_{(n-1)\varphi,\Gamma} \, \phi = 0 & \mbox{ on }\Gamma,\\
        \phi = z &\mbox{ on }\partial \Gamma.
    \end{cases}
\end{align*}
By choosing the appropriate component as $\Omega$,
we may assume 
\begin{align}\label{eq:assumption-ii}
    \int_\Sigma \mathrm{II}_{\Sigma}(\nabla_{\Sigma}\,z, \nabla_{\Sigma}\,z)\ \d \mu_{\Sigma} \geq 0,
\end{align}
where $\mu := \e^{-(n-1)\varphi}v_g$.
Indeed,
when this inequality does not hold,
we choose the other component as $\Omega$.
Let $u$ be the solution of the following boundary value problem:
\begin{align*}
    \begin{cases}
        \Delta u - n g( \nabla \varphi, \nabla u ) = 0& \mbox{ on }\Omega,\\
        u = z&\mbox{ on } \Sigma,\\
        u = \phi&\mbox{ on } \Gamma.
    \end{cases}
\end{align*}
By the assumption \eqref{eq:assumption-ii} and Proposition \ref{prop:d-reilly},
we have 
\begin{align*}
    0&\geq - \int_{\Sigma}  g_\Sigma \left( \nabla_\Sigma\, z, \nabla_\Sigma \, u_{\nu(\Sigma)} \right) \ \d \mu_{\Sigma} - \int_{\Gamma}  g_\Gamma \left( \nabla_\Gamma \phi, \nabla_\Gamma \,u_{\nu(\Gamma)} \right)\ \d \mu_{\Gamma} + k\int_\Gamma  | \nabla_\Gamma \phi|^2 \ \d \mu_{\Gamma}\\
    &\geq - \int_{\partial\Sigma} z_{\nu(\partial\Sigma)}\, u_{\nu(\Sigma)}\ \d \mu_{\partial\Sigma} - \int_{\partial\Gamma} \phi_{\nu(\partial\Gamma)}\, u_{\nu(\Gamma)}\ \d \mu_{\partial\Gamma} + k\int_\Gamma |\nabla_{\Gamma} \,\phi|^2 \ \d \mu_{\Gamma}.
\end{align*}
From the free boundary condition, 
we have $\nu(\Sigma) = \nu(\partial\Gamma)$ and $\nu(\Gamma) = \nu(\partial \Sigma)$.
Hence,
\begin{align*}
    k\int_\Gamma |\nabla_\Gamma \phi|^2 \ \d \mu 
    &\leq 2 \int_{\partial\Gamma}\phi_{\nu(\partial\Gamma)} \, z_{\nu(\partial\Sigma)} \ \d \mu_{\partial\Gamma}\\
    &= 2 \lambda_{1,\Sigma}^{\mathrm{Ste}} \int_{\partial\Gamma} \phi_{\nu(\partial\Gamma)} \phi \ \d \mu_{\partial\Gamma}\\
    &= 2 \lambda_{1,\Sigma}^{\mathrm{Ste}} \int_{\Gamma} |\nabla_\Gamma \phi|^2 \ \d \mu_{\Gamma}.
\end{align*}
This implies 
\begin{align*}
    \lambda_{1,\Sigma}^\mathrm{Ste} \geq \frac{k}{2}.
\end{align*}

We next consider the case $\pi_1(M)$ is finite.
Let $\overline{M}$ be the universal cover of $M$.
Let $\overline{\Sigma}$ and $\overline{\varphi}$ be the lift of $\Sigma$ and $\varphi$.
Then $\overline{\Sigma}$ is orientable.
Therefore,
we may apply the argument above to $\overline{\Sigma}$,
and we see that the first $(n-1)\overline{\varphi}$-Steklov eigenvalue on $\overline{\Sigma}$ satisfies
\begin{align*}
    \lambda^{\mathrm{Ste}}_{1,\overline{\Sigma}} \geq \frac{k}{2}.
\end{align*}
Combining this with $\lambda_{1,\Sigma}^{\mathrm{Ste}} \geq \lambda^{\mathrm{Ste}}_{1,\overline{\Sigma}}$,
we obtain the desired result.
\end{proof}
\begin{remark}
We refer to Fraser-Li \cite[Theorem 3.1]{FL} for the unweighted case $f\equiv 0$ and Barbosa-Wei \cite[Proposition 3.1]{BW} for the weighted case with $N = \infty$.
This is the generalization of them to the case $N = 1$.
\end{remark}
We now turn to a Steklov type boundary value problem in Huang-Ma-Zhu \cite{HMZ}.
In \cite[(1.13)]{HMZ},
they considered $((M,g),V)$ and a boundary value problem:
\begin{align}\label{eq:HMZ-steklov-boundary-value}
    \begin{cases}
        \Delta_\Sigma\, u + 2 g_\Sigma (\nabla_\Sigma\,\varphi, \nabla_\Sigma \,u) = 0 &\mbox{ on }\Sigma,\\
        \e^{\varphi} u_\nu = \eta \, u &\mbox{ on } \partial \Sigma,
    \end{cases}
\end{align}
where $\varphi := \log V$ and $\nu := \nu(\partial \Sigma)$.
We denote the first eigenvalue by $\eta_{1,\Sigma}^{\mathrm{Ste}}$.
By direct calculations,
we see that $\eta_{1,\Sigma}^{\mathrm{Ste}}$ coincides with the first $\{-(n-1)\varphi\}$-Steklov eigenvalue in \eqref{eq:steklov-problem} on $(\Sigma, \widetilde{g}_\Sigma)$ with $\widetilde{g} := \e^{-2\varphi} g$.
Indeed,
we have the following relation:
\begin{align*}
    \begin{cases}
        \widetilde{\Delta}_{\Sigma, -(n-1)\varphi} \,u = \e^{2\varphi}\left\{\Delta_\Sigma \, u + 2g_\Sigma(\nabla_\Sigma\,\varphi, \nabla_\Sigma \,u) \right\} &\mbox{ on } \Sigma, \\
        \widetilde{g}_\Sigma (\widetilde{\nabla}_\Sigma \, u, \widetilde{\nu}) = \e^\varphi u_\nu &\mbox{ on } \partial \Sigma,
    \end{cases}
\end{align*}
where $\widetilde{\Delta}$ is the Laplacian for $\widetilde{g}$ and $\widetilde{\nu} := \e^\varphi \nu$.
Hence,
as an immediate application of Theorem \ref{thm:steklov-1-ricci},
we have the following lower bound of $\eta_{1,\Sigma}^{\mathrm{Ste}}$:
\begin{corollary}\label{cor:steklov-substatic}
Let $((M,g),V)$ be a compact substatic triple with boundary.
For $k > 0$,
we assume 
\begin{align}\label{eq:steklov-condition}
    V \mathrm{II}_{\partial M} - V_\nu\, g_{\partial M} \geq k \,g_{\partial M},\quad H_{\partial M} > 0,
\end{align}
where $\nu := \nu(\partial M)$.
Let $\Sigma$ be a properly embedded minimal hypersurface in $M$ with free boundary and $\eta_{1,\Sigma}^{\mathrm{Ste}}$ be the first Steklov type eigenvalue in \eqref{eq:HMZ-steklov-boundary-value} for $\varphi := \log V$.
If $\Sigma$ is orientable or $\pi_1(M)$ is finite,
we have
\begin{align*}
    \eta_{1,\Sigma}^{\mathrm{Ste}} \geq \frac{k}{2}.
\end{align*}
\end{corollary}
\begin{remark}
The condition in \eqref{eq:steklov-condition} also appeared in Huang-Ma-Zhu \cite[Theorem 1.3]{HMZ}.
\end{remark}
We are now in a position to give the following compactness theorem:
\begin{theorem}\label{thm:compactness}
Let $(M,g)$ be a three-dimensional compact Riemannian manifold with boundary and $\varphi \in C^{\infty}(M)$.
For $D := \nabla^\varphi$ and $k > 0$,
we assume 
\begin{align*}
    \Ric^D \geq 0, \quad \mathrm{II}_{\partial M} \geq k\,g_{\partial M}, \quad H^D_{\partial M} > 0.
\end{align*}
Let $\mathcal{S}$ be the space of compact properly embedded $D$-minimal surfaces of fixed topological type.
Then $\mathcal{S}$ is compact in the $C^l$-topology for any $l\geq 2$.
\end{theorem}
\begin{proof}
By Corollary \ref{cor:diffeo-ball},
we see that $M$ is diffeomorphic to a three-dimensional ball.
This implies that $M$ is simply connected.
For $\Sigma \in \mathcal{S}$ and $\widetilde{g} := \e^{-2\varphi} g$,
we denote the sectional curvature on $M$ by $\mathrm{Sec}_{\widetilde{g}}$,
and the second fundamental form on $\Sigma$ by $\widetilde{\mathrm{II}}_\Sigma$,
and also the geodesic curvature on $\partial \Sigma$ by $\kappa_{\widetilde{g},\partial \Sigma}$.
It follows from the argument in \cite[Theorem 1.3]{BW} that 
\begin{align}\label{eq:bw-thm-1-3}
    \frac{1}{2}\int_\Sigma |\widetilde{\mathrm{II}}_\Sigma|^2 \ \d v_{\widetilde{g},\Sigma} = \int_\Sigma \mathrm{Sec}_{\widetilde{g}} \ \d v_{\widetilde{g},\Sigma} + \int_{\partial \Sigma} \kappa_{\widetilde{g},\partial \Sigma}\ \d v_{\widetilde{g},\partial\Sigma} - 2\pi (2- \alpha - \mathrm{gen}(\Sigma)),
\end{align}
where $\mathrm{gen}(\Sigma)$ is the number of genus of $\Sigma$ and $\alpha$ is the number of components of $\partial \Sigma$.
Together with Theorem \ref{thm:FL-isoperimetric},
the right-hand side of \eqref{eq:bw-thm-1-3} is estimated as follows:
\begin{align*}
    \int_\Sigma \mathrm{Sec}_{\widetilde{g}} \ \d v_{\widetilde{g},\Sigma} &\leq C_1 v_{\widetilde{g},\Sigma}(\Sigma) \leq C_2 v_{\widetilde{g},\partial \Sigma}(\partial \Sigma),\quad \int_{\partial \Sigma} \kappa_{\widetilde{g},\partial \Sigma}\ \d v_{\widetilde{g},\Sigma} \leq C_3 v_{\widetilde{g},\partial \Sigma}(\partial \Sigma),
\end{align*}
where $C_1, C_2, C_3$ are positive constants depending only on the geometry of $(M,g)$ and $\|\varphi\|_{C^2}$.
From the argument in \cite[Corollary 3.2]{BW} and Theorem \ref{thm:steklov-1-ricci},
we have
\begin{align*}
    v_{\widetilde{g},\partial \Sigma}(\partial \Sigma) \leq \frac{2\pi(\mathrm{gen}(\Sigma) + \alpha)}{\lambda^{\mathrm{Ste}}_{1,\Sigma}}\e^{7 \max \varphi} \leq \frac{4\pi(\mathrm{gen}(\Sigma) + \alpha)}{k}\,\e^{7\max \varphi}.
\end{align*}
Combining these with \eqref{eq:bw-thm-1-3},
we see that $\int_\Sigma |\widetilde{\mathrm{II}}_\Sigma|^2 \ \d v_{\widetilde{g},\Sigma}$ is bounded from above by a constant depending only on the topology of $\Sigma$,
the geometry of $(M,g)$ and $\|\varphi\|_{C^2}$.
Therefore,
for any sequence in $\mathcal{S}$,
the argument in \cite[Theorem 1.3]{BW} yields that there exists a subsequence $\{\Sigma_i\}$,
and a finite set of points $\mathcal{N}$ such that $\{\Sigma_i\}$ converges in smooth topology to a surface $\Sigma$ off of $\mathcal{N}$.
We may consider that the limit $\Sigma$ is properly embedded $D$-minimal surface with free boundary by the removal of singularity theorem in \cite[Theorem 4.1]{FL}.
If the multiplicity of the convergence is one,
then the Allard regularity theorem in \cite{GJ} with free boundary implies that the convergence is smooth everywhere even across $\mathcal{N}$.
If the multiplicity is greater than one,
as is constructed in \cite[Theorem 1.3]{BW},
there exists a function $\phi_i$ on each $\Sigma_i$ such that 
\begin{align*}
    \lambda_{1,\Sigma_i}^{\mathrm{Ste}} \leq \frac{\int_{\Sigma_i} |\nabla_{\Sigma_i} \phi_i|^2 \ \d \mu_{\Sigma_i}}{\int_{\partial \Sigma_i}\phi_i^2 \ \d \mu_{\partial \Sigma_i}}\rightarrow 0
\end{align*}
as $i\rightarrow \infty$,
where $\mu := \e^{-(n-1)\varphi}v_g$.
This contradicts with Theorem \ref{thm:steklov-1-ricci}.
We arrive at the desired assertion.
\end{proof}
\begin{remark}
We refer to Fraser-Li \cite[Theorem 6.1]{FL} for the unweighted case $f \equiv 0$ and Barbosa-Wei \cite[Theorem 1.3]{BW} for the weighted case with $N = \infty$.
This is the generalization of them to the case $N = 1$.
\end{remark}
Lastly,
we write down an application to substatic triples:
\begin{corollary}\label{cor:substatic-compactness}
Let $((M,g),V)$ be a three-dimensional compact substatic triple with boundary.
For $k > 0$,
we assume 
\begin{align*}
    V \mathrm{II}_{\partial M} - V_\nu\, g_{\partial M} \geq k \,g_{\partial M},\quad H_{\partial M} > 0,
\end{align*}
where $\nu := \nu(\partial M)$.
Let $\mathcal{S}$ be the space of compact properly embedded minimal surfaces of fixed topological type.
Then $\mathcal{S}$ is compact in the $C^l$-topology for any $l\geq 2$.
\end{corollary}
\subsection*{{\rm{Acknowledgements}}}

The author would like to express gratitude to Professor Yohei Sakurai and Ryu Ueno for fruitful discussions.
Also,
the authour thank Professor Shin-ichi Ohta for valuable comments.
The work was supported by JST, the establishment of university fellowships towards the creation of science technology innovation, Grant Number JPMJFS2125.

\bibliographystyle{amsplain}
\bibliography{ref} 

\begin{bibdiv}
\begin{biblist}

\bib{BW}{article}{
      author={Barbosa, Ezequiel},
      author={Wei, Yong},
       title={{A compactness theorem of the space of free boundary $f$-minimal
  surfaces in three-dimensional smooth metric measure space with boundary}},
        date={2016},
        ISSN={1050-6926,1559-002X},
     journal={J. Geom. Anal.},
      volume={26},
      number={3},
       pages={1995\ndash 2012},
         url={https://doi.org/10.1007/s12220-015-9616-4},
}

\bib{BF}{article}{
      author={Borghini, Stefano},
      author={Fogagnolo, Mattia},
       title={{Comparison geometry for substatic manifolds and a weighted
  Isoperimetric Inequality}},
        date={2023},
     journal={arXiv preprint:2307.14618},
}

\bib{B3}{article}{
      author={Brendle, Simon},
       title={Constant mean curvature surfaces in warped product manifolds},
        date={2013},
        ISSN={0073-8301,1618-1913},
     journal={Publ. Math. Inst. Hautes \'Etudes Sci.},
      volume={117},
       pages={247\ndash 269},
         url={https://doi.org/10.1007/s10240-012-0047-5},
}

\bib{CR}{article}{
      author={Castro, Katherine},
      author={Rosales, C\'esar},
       title={Free boundary stable hypersurfaces in manifolds with density and
  rigidity results},
        date={2014},
        ISSN={0393-0440,1879-1662},
     journal={J. Geom. Phys.},
      volume={79},
       pages={14\ndash 28},
         url={https://doi.org/10.1016/j.geomphys.2014.01.013},
}

\bib{CMZ}{article}{
      author={Cheng, Xu},
      author={Mejia, Tito},
      author={Zhou, Detang},
       title={{Eigenvalue estimate and compactness for closed $f$-minimal
  surfaces}},
        date={2014},
        ISSN={0030-8730,1945-5844},
     journal={Pacific J. Math.},
      volume={271},
      number={2},
       pages={347\ndash 367},
         url={https://doi.org/10.2140/pjm.2014.271.347},
}

\bib{CS}{article}{
      author={Choi, Hyeong~In},
      author={Schoen, Richard},
       title={The space of minimal embeddings of a surface into a
  three-dimensional manifold of positive {R}icci curvature},
        date={1985},
        ISSN={0020-9910,1432-1297},
     journal={Invent. Math.},
      volume={81},
      number={3},
       pages={387\ndash 394},
         url={https://doi.org/10.1007/BF01388577},
}

\bib{CW}{article}{
      author={Choi, Hyeong~In},
      author={Wang, Ai~Nung},
       title={{A first eigenvalue estimate for minimal hypersurfaces}},
        date={1983},
        ISSN={0022-040X,1945-743X},
     journal={J. Differential Geom.},
      volume={18},
      number={3},
       pages={559\ndash 562},
         url={http://projecteuclid.org/euclid.jdg/1214437788},
}

\bib{CM3}{article}{
      author={Colding, Tobias~H.},
      author={Minicozzi, William~P., II},
       title={Smooth compactness of self-shrinkers},
        date={2012},
        ISSN={0010-2571,1420-8946},
     journal={Comment. Math. Helv.},
      volume={87},
      number={2},
       pages={463\ndash 475},
         url={https://doi.org/10.4171/CMH/260},
}

\bib{DX}{article}{
      author={Ding, Qi},
      author={Xin, Y.~L.},
       title={Volume growth, eigenvalue and compactness for self-shrinkers},
        date={2013},
        ISSN={1093-6106,1945-0036},
     journal={Asian J. Math.},
      volume={17},
      number={3},
       pages={443\ndash 456},
         url={https://doi.org/10.4310/AJM.2013.v17.n3.a3},
}

\bib{FLZ}{article}{
      author={Fang, Fuquan},
      author={Li, Xiang-Dong},
      author={Zhang, Zhenlei},
       title={Two generalizations of {C}heeger-{G}romoll splitting theorem via
  {B}akry-{E}mery {R}icci curvature},
        date={2009},
        ISSN={0373-0956,1777-5310},
     journal={Ann. Inst. Fourier (Grenoble)},
      volume={59},
      number={2},
       pages={563\ndash 573},
         url={https://doi.org/10.5802/aif.2440},
}

\bib{FL}{article}{
      author={Fraser, Ailana},
      author={Li, Martin},
       title={Compactness of the space of embedded minimal surfaces with free
  boundary in three-manifolds with nonnegative {R}icci curvature and convex
  boundary},
        date={2014},
        ISSN={0022-040X,1945-743X},
     journal={J. Differential Geom.},
      volume={96},
      number={2},
       pages={183\ndash 200},
         url={http://projecteuclid.org/euclid.jdg/1393424916},
}

\bib{FS}{article}{
      author={Fujitani, Yasuaki},
      author={Sakurai, Yohei},
       title={{Geometric analysis on weighted manifolds under lower $0
  $-weighted Ricci curvature bounds}},
        date={2024},
     journal={arXiv preprint:2408.15744},
}

\bib{GJ}{article}{
      author={Gr\"uter, Michael},
      author={Jost, J\"urgen},
       title={Allard type regularity results for varifolds with free
  boundaries},
        date={1986},
        ISSN={0391-173X,2036-2145},
     journal={Ann. Scuola Norm. Sup. Pisa Cl. Sci. (4)},
      volume={13},
      number={1},
       pages={129\ndash 169},
         url={http://www.numdam.org/item?id=ASNSP_1986_4_13_1_129_0},
}

\bib{HMZ}{article}{
      author={Huang, Guangyue},
      author={Ma, Bingqing},
      author={Zhu, Mingfang},
       title={Colesanti type inequalities for affine connections},
        date={2023},
        ISSN={1664-2368,1664-235X},
     journal={Anal. Math. Phys.},
      volume={13},
      number={1},
       pages={Paper No. 12, 15},
         url={https://doi.org/10.1007/s13324-022-00773-8},
}

\bib{LW}{article}{
      author={Li, Haizhong},
      author={Wei, Yong},
       title={{$f$}-minimal surface and manifold with positive
  {$m$}-{B}akry-\'{E}mery {R}icci curvature},
        date={2015},
        ISSN={1050-6926,1559-002X},
     journal={J. Geom. Anal.},
      volume={25},
      number={1},
       pages={421\ndash 435},
         url={https://doi.org/10.1007/s12220-013-9434-5},
}

\bib{LX}{article}{
      author={Li, Junfang},
      author={Xia, Chao},
       title={An integral formula for affine connections},
        date={2017},
        ISSN={1050-6926,1559-002X},
     journal={J. Geom. Anal.},
      volume={27},
      number={3},
       pages={2539\ndash 2556},
         url={https://doi.org/10.1007/s12220-017-9771-x},
}

\bib{LMM}{article}{
      author={Li, Martin Man-chun},
       title={A general existence theorem for embedded minimal surfaces with
  free boundary},
        date={2015},
        ISSN={0010-3640,1097-0312},
     journal={Comm. Pure Appl. Math.},
      volume={68},
      number={2},
       pages={286\ndash 331},
         url={https://doi.org/10.1002/cpa.21513},
}

\bib{L}{article}{
      author={Lichnerowicz, Andr\'{e}},
       title={Vari\'{e}t\'{e}s riemanniennes \`a tenseur {C} non n\'{e}gatif},
        date={1970},
        ISSN={0151-0509},
     journal={C. R. Acad. Sci. Paris S\'{e}r. A-B},
      volume={271},
       pages={A650\ndash A653},
}

\bib{MD}{article}{
      author={Ma, Li},
      author={Du, Sheng-Hua},
       title={Extension of {R}eilly formula with applications to eigenvalue
  estimates for drifting {L}aplacians},
        date={2010},
        ISSN={1631-073X,1778-3569},
     journal={C. R. Math. Acad. Sci. Paris},
      volume={348},
      number={21-22},
       pages={1203\ndash 1206},
         url={https://doi.org/10.1016/j.crma.2010.10.003},
}

\bib{MSY}{article}{
      author={Meeks, William, III},
      author={Simon, Leon},
      author={Yau, Shing~Tung},
       title={Embedded minimal surfaces, exotic spheres, and manifolds with
  positive {R}icci curvature},
        date={1982},
        ISSN={0003-486X},
     journal={Ann. of Math. (2)},
      volume={116},
      number={3},
       pages={621\ndash 659},
         url={https://doi.org/10.2307/2007026},
}

\bib{WW}{article}{
      author={Wei, Guofang},
      author={Wylie, Will},
       title={Comparison geometry for the {B}akry-{E}mery {R}icci tensor},
        date={2009},
        ISSN={0022-040X,1945-743X},
     journal={J. Differential Geom.},
      volume={83},
      number={2},
       pages={377\ndash 405},
         url={https://doi.org/10.4310/jdg/1261495336},
}

\bib{White2}{article}{
      author={White, Brian},
       title={Which ambient spaces admit isoperimetric inequalities for
  submanifolds?},
        date={2009},
        ISSN={0022-040X,1945-743X},
     journal={J. Differential Geom.},
      volume={83},
      number={1},
       pages={213\ndash 228},
         url={http://projecteuclid.org/euclid.jdg/1253804356},
}

\bib{W1}{article}{
      author={Wylie, William},
       title={Sectional curvature for {R}iemannian manifolds with density},
        date={2015},
        ISSN={0046-5755,1572-9168},
     journal={Geom. Dedicata},
      volume={178},
       pages={151\ndash 169},
         url={https://doi.org/10.1007/s10711-015-0050-3},
}

\bib{W2}{article}{
      author={Wylie, William},
       title={A warped product version of the {C}heeger-{G}romoll splitting
  theorem},
        date={2017},
        ISSN={0002-9947,1088-6850},
     journal={Trans. Amer. Math. Soc.},
      volume={369},
      number={9},
       pages={6661\ndash 6681},
         url={https://doi.org/10.1090/tran/7003},
}

\bib{WY}{article}{
      author={Wylie, William},
      author={Yeroshkin, Dmytro},
       title={{On the geometry of Riemannian manifolds with density}},
        date={2016},
     journal={arXiv preprint:1602.08000},
}

\bib{Yero}{article}{
      author={Yeroshkin, Dmytro},
       title={Holonomy of manifolds with density},
        date={2020},
     journal={arXiv preprint:2009.08733},
}

\end{biblist}
\end{bibdiv}
\end{document}